\documentclass[10pt,a4paper]{amsart}
\usepackage{amsfonts}
\usepackage{amsthm}
\usepackage{amsmath}
\usepackage{amscd}
\usepackage[latin2]{inputenc}
\usepackage{t1enc}
\usepackage[mathscr]{eucal}
\usepackage{indentfirst}
\usepackage{graphicx}
\usepackage{graphics}
\usepackage{pict2e}
\usepackage{epic}
\usepackage[normalem]{ulem}
\numberwithin{equation}{section}
\usepackage{epstopdf}
\usepackage{verbatim}
\usepackage{amssymb}
\usepackage{tikz-cd}
\usepackage{mathtools} 
\usepackage{extarrows}

\usepackage{hyperref}
\AtBeginDocument{}

\theoremstyle{plain}
\newtheorem{Th}{Theorem}[section]
\newtheorem{Lemma}[Th]{Lemma}
\newtheorem{Cor}[Th]{Corollary}
\newtheorem{Prop}[Th]{Proposition}

\theoremstyle{definition}
\newtheorem{Def}[Th]{Definition}

\newtheorem{Rem}[Th]{Remark}
\newtheorem{?}[Th]{Problem}

\def\al{\alpha}

\def\w{\wedge}
\def\R{\mathbb{R}}
\def\C{\mathbb{C}}

\def\Lm{\Lambda}

\def\om{\omega}
\def\Om{\Omega}
\def\vp{\varphi}
\def\ddt{\frac{\partial}{\partial t}}

\def\ddu{\frac{\partial}{\partial u}}

\def\ip{\raise1pt\hbox{\large$\lrcorner$}\>}

\DeclareMathOperator{\vol}{vol}

\DeclareMathOperator{\Ric}{Ric}

\begin{document}
	
\title{$S^1$-invariant Laplacian flow}

\author[U. Fowdar]{Udhav Fowdar}
	
\address{University College London \\ Department of Mathematics \\
Gower Street \\
WC1E  6BT\\
London \\
UK}
	
\email{udhav.fowdar.12@ucl.ac.uk}
		
	


\begin{abstract} 
The Laplacian flow is a geometric flow introduced by Bryant as a way for finding torsion free $G_2$-structures starting from a closed one. If the flow is invariant under a free $S^1$ action then it descends to a flow of $SU(3)$-structures on a $6$-manifold. In this article we derive expressions for these evolution equations. In our search for examples we discover the first inhomogeneous shrinking solitons, which are also gradient. We also show that any compact non-torsion free soliton admits no infinitesimal symmetry.
\end{abstract}
	
\maketitle
\tableofcontents

\section{Introduction}
\subsection{Overview.} A symplectic $SU(3)$-structure on a $6$-manifold $P^6$ is given by a pair $(\om,\Om)$, where $\om$ is a symplectic form and $\Om$ is a complex $(3,0)$-form with real and imaginary parts given by $\Om^+$ and $\Om^-$ respectively.
Our main goal in this article is to derive the evolution equations for such an $SU(3)$-structure on $P^6$,
together with a Higgs fields $H:P^6 \to \R^+$ and a connection $1$-form $\xi$ on an $S^1$ bundle $L^7 \to P^6$, such that 
\[\vp(t) :=\xi(t) \w \om(t) + H(t)^{3/2} \Om^+(t) \]
is a solution to the Laplacian flow (see below for the definition). Put differently we want to understand how the $SU(3)$-structure on the quotient of an $S^1$-invariant solution to the Laplacian flow evolves. 
This investigation is further motivated from the results of Lotay-Wei in \cite{Lotay2017} that assert that if the initial $3$-form $\vp_0$ on a compact manifold $L^7$ is $S^1$-invariant then so is the flow for as long as it exists.
We study two simple cases and as a result we construct the first examples of  inhomogeneous shrinking Laplacian solitons. It is known that shrinking solitons do not occur on compact manifolds and that the steady ones are actually torsion free \cite{Lotay2017}. Indeed the examples that we construct are non-compact. There are currently no known compact examples of expanding solitons. We prove that such expanders, if they exist at all, in fact do not admit any infinitesimal symmetry i.e. there are no Killing vector fields which preserve the associated $G_2$ $3$-form $\vp$. This gives a partial explanation why such solitons have not (yet) been found.

\subsection{Motivation.} A closed $G_2$-structure on a $7$-manifold $L^7$ is given by a closed non-degenerate $3$-form $\vp$, which in turn determines  a Riemannian metric $g_\vp$ and volume form $\vol_\vp$. In particular, $\vp$ also defines a Hodge star operator $*_\vp$. If $*_\vp\vp$ is also closed then the holonomy of $g_\vp$ is a subgroup of $G_2$ and consequently the metric $g_\vp$ is Ricci-flat. The Laplacian flow, defined as the initial value problem
\begin{align}
\ddt \vp(t) &= \mathrm{\Delta}_{\vp(t)}\vp(t),\label{laplacianequa}\\
\vp(0) &=\vp_0,\\
d\vp_0&=0 \label{laplacianequa3}
\end{align}
where $ \mathrm{\Delta}_{\vp} := dd^{*_\vp}+d^{*_\vp} d$, was introduced by Bryant in \cite{Bryant06someremarks} as a way of potentially deforming $\vp_0$ within its cohomology class to a torsion free one. The short time existence of the flow was proved by Bryant and Xu in \cite{BryantXu2011}, and long time existence, uniqueness and compactness results were proved by Lotay and Wei in \cite{Lotay2017}. 

In recent years there have been many works on the Laplacian flow cf. \cite{Fine2018, Fino2017, Fino2020, LotayLambert, Lauret2017, Nicolini2020} but the flow is nonetheless still very complicated to study in its full generality. A natural strategy to simplifying the equations is to impose symmetry. 
Aside from the homogeneous cases, three notable works in this direction were carried out by Fine and Yao in \cite{Fine2018}, where the authors, motivated by a question of Donaldson, study the flow on $L^7=M^4\times \mathbb{T}^3$ as a way of deforming a hypersymplectic triple on $M^4$ to a hyperK\"ahler one,
by Lotay and Lambert in \cite{LotayLambert}, where the authors reinterpret the flow on $L^7=B^3 \times \mathbb{T}^4$ as a spacelike mean curvature flow on $B^3\subset\R^{3,3}$,
and by Fino and Raffero in \cite{Fino2017}, where the authors study the flow for warped $G_2$-structures on $L^7=P^6\times S^1$. In each case the presence of a free abelian group action allows for a reduction of the Laplacian flow to a lower dimensional space. It is exactly this common feature that motivates the work undertaken here.

\subsection{Outline.} In this article we consider the more general case when $L^7$ only admits a free $S^1$ action. In this case the Laplacian flow descends to a flow on a symplectic $SU(3)$-structure on $P^6$, also called almost K\"ahler, together with the data of a Higgs field and a connection $1$-form. This leads us to try and understand the resulting flow. In section \ref{quotientofclosedg2structures} we prove a Gibbons-Hawking type construction for $S^1$-invariant closed $G_2$-structures (see Theorem \ref{GHcalibratedG2}) and we express the intrinsic torsion of the $G_2$-structure in terms of that of the $SU(3)$-structure. In other words we encode that data of an $S^1$-invariant closed $G_2$-structure only in terms of data on the quotient. In section \ref{LFsection} we then derive the evolution equations for such data when $\vp$ is evolving by the Laplacian flow. We also prove that certain cohomology classes have to vanish on a compact manifold with an exact $G_2$-structure (see Theorem \ref{topoexact}) and as a consequence we deduce that compact expanding solitons cannot arise from our construction. Even if $(L^7,\vp)$ is torsion free the $SU(3)$-structure on the quotient $P^6:=L^7/S^1$ is generally not torsion free, though it is always symplectic. For a flow of $SU(3)$-structures a natural question one can ask is if any class of $SU(3)$-structure is preserved. For instance in \cite{Fino2017}, the authors show that under certain constraints if $(\om,\Om)$ is initially symplectic half-flat then this is preserved by the flow. In section \ref{examples} we first consider the original example of the Laplacian flow studied by Bryant and show that the symplectic form on the quotient, for a given $S^1$ action, in fact remains constant under the flow, though the intrinsic torsion of the quotient $SU(3)$-structure is generic. Motivated by the work of Apostolov and Salamon in \cite{Apostolov2003} where the authors construct (incomplete) examples of $G_2$ metrics which admit K\"ahler reduction, we next search for solutions to the Laplacian flow on these spaces. We show that the flow indeed preserves the K\"ahler property in this setting and we find explicit (complete) shrinking gradient solitons. These are the only currently known examples of inhomogeneous shrinkers.\\ 

\addtocontents{toc}{\protect\setcounter{tocdepth}{-1}}
\noindent\textbf{Acknowledgements.}
The author is indebted to his PhD advisors Jason Lotay and Simon Salamon for their constant support and many helpful discussions that led to this article. The author would also like to thank Andrew Dancer and Lorenzo Foscolo for helpful comments on a version of this paper that figures in the author's thesis. This work was supported by the Engineering and Physical Sciences Research Council [EP/L015234/1], The EPSRC Centre for Doctoral Training in Geometry and Number Theory (The London School of Geometry and Number Theory), University College London. 
\addtocontents{toc}{\protect\setcounter{tocdepth}{1}}

\section{Preliminaries}
The aim of this section is mainly to give a concise introduction to the basic objects of interest in this article and to set up notation. 
\subsection{$SU(3)$-structures}
\begin{Def}
An $SU(3)$-structure on a $6$-manifold $P^6$ consists of an almost Hermitian structure determined by an almost complex structure $J$, a non-degenerate $2$-form $\om$ and an almost Hermitian metric $g_\om$ satisfying
\begin{equation}
g_\om(\cdot,\cdot)=\om(\cdot ,J\cdot), \label{compatibilitygomega}
\end{equation}
and a $(3,0)$-form $\Om:=\Om^++i\Om^-$ satisfying the two conditions
\begin{gather} 
\om \w \Om^\pm=0, \label{compatibilitysu3}\\
\frac{2}{3} \om^3={\Om^+ \w \Om^-} .\label{normalisationsu3}
\end{gather}
\end{Def}
Note that condition (\ref{compatibilitysu3}) is just a consequence of the fact that $\om$ is of type $(1,1)$.
Although an $SU(3)$-structure consists of the data $(g_\om,\om,J,\Om)$, one can in fact recover the whole $SU(3)$-structure only from the pair $(\om,\Om^+)$, or $(\om,\Om^-)$. This observation is due to Hitchin in \cite{Hitchin2000} where he shows that $\Om^+$, or $\Om^-$, determines $J$. Abstractly this follows from the fact that the stabiliser of $\Om^\pm$ in $GL^+(6,\R)$ is congruent to $SL(3,\C)\subset GL(3,\C)$. The metric is then determined by (\ref{compatibilitygomega})
and $\Om^-=J(\Om^+)=*_\om\Om^+$, where $*_\om$ is the Hodge star operator determined by $g_\om$ and the volume form
\begin{equation}
\vol_\om:=\frac{1}{6}\om^3=\frac{1}{4}\Om^+\w\Om^-.\end{equation}
As $SU(3)$ modules the space of exterior differential forms splits as
\begin{align}
\Lm^1 &= \Lm^1_6=[\![ \Lm^{1,0}]\!]\nonumber\\
\Lm^2 &= \langle\om\rangle \oplus\Lm^2_6 \oplus \Lm^2_{8}\label{decomposition2form}\\
\Lm^3 &= \langle\Om^+\rangle\oplus \langle\Om^-\rangle \oplus \Lm^3_6 \oplus \Lm^3_{12}\nonumber
\end{align}
where
\begin{align}
\Lm^2_6=[\![ \Lm^{2,0}]\!]&=\{\al\in \Lm^2\ |\ *_\om(\al \w \om)=\al\}\label{lambda26}\\
&=\{*_\om(\al\w\Om^+) \ |\ \al\in\Lm^1_6 \}\label{lambda262}\\
\Lm^2_{8}=[ \Lm^{1,1}_0]&=\{\al\in \Lm^2\ |\ *_\om(\al \w \om)=-\al\}\label{lambda28}\\
&=\{\al\in \Lm^2\ |\ \al \w \om^2=0 \ \ \text{and} \ \ \al \w \Om^+=0 \}\label{lambda282}\\
\Lm^3_6=[\![ \Lm^{2,1}]\!]&=\{\al\w \om\ |\ \al \in \Lm^1_6\}\nonumber\\
\Lm^3_{12}=[\![ \Lm^{2,1}_0]\!]&=\{\al\in \Lm^3\ |\ \al \w \om=0,\ \al \w \Om^{\pm}=0 \}\nonumber
\end{align}
and we get a corresponding decomposition for $\Lm^4$ and $\Lm^5$ via $*_\om$. The $[\![\Lm^{p,q}]\!]$ and $[\Lm^{p,p}]$ notation refers to taking the corresponding \textit{real} underlying vector space i.e. $$[\![\Lm^{p,q}]\!]\otimes \C=\Lm^{p,q}\oplus\Lm^{q,p}\text{\ \ and\ \ }[\Lm^{p,p}]\otimes \C=\Lm^{p,p}$$ cf. \cite{ChiossiSalamonIntrinsicTorsion, Salamon1989}. As $SU(3)$ modules the spaces $\Lm^\bullet_6$ are all isomorphic. In computations we will often need to interchange between these spaces and to do so we use the following lemma which follows from a simple calculation. 
\begin{Lemma}\label{interchangelemma}
Given a $1$-form $\al\in\Lm^1_6$, let $\beta:=*_\om(\al \w \Om^-)\in \Lm^2_6$. Then the following hold:
\begin{enumerate}
\item $J(\al)\w\Om^+ = \al \w \Om^-=\beta\w\om$
\item $\beta \w \Om^-=2*_\om(\al)=-(J\al)\w \om^2$
\item $\beta \w \Om^+=2*_\om(J\al)=\al\w\om^2$
\end{enumerate}
\end{Lemma}
\noindent We follow the convention that $(J\al)(v):=\al(Jv)$ for a $1$-form $\al$ and vector $v$, which might differ by a minus sign from other conventions in the literature. We also define the operator $d^cf:=J\circ df$ on a real function $f$.

The intrinsic torsion of the $SU(3)$-structure $(\om,\Om^+)$ is determined by
\begin{align}
d\om&=-\frac{3}{2}\sigma_0\ \Om^++\frac{3}{2}\pi_0\ \Om^-+\nu_1\w \om+\nu_3,\label{su3torsion1}\\
d\Om^+&=\pi_0\ \om^2+\pi_1\w\Om^+-\pi_2\w\om,\label{su3torsion2}\\
d\Om^-&=\sigma_0\ \om^2+(J\pi_1)\w\Om^+-\sigma_2\w\om\label{su3torsion3},
\end{align}
where $\sigma_0,\pi_0$ are functions, $\nu_1,\pi_1\in\Lm^1_6$, $\pi_2,\sigma_2\in\Lm^2_8$ and $\nu_3 \in \Lm^3_{12}$, cf. \cite{Bedulli2007}. Many well-known geometric structures can be recast using this formulation:
\begin{Def}\label{su3strucdefinitions}
	The $SU(3)$-structure $(\om,\Om^+)$	is said to be
	\begin{enumerate}
		\item Calabi-Yau if  all the torsion forms vanish,\label{CYdefinition}
		\item almost-K\"ahler if  $\sigma_0=\pi_0=\nu_1=\nu_3=0$,
		\item K\"ahler if all the torsion forms aside from $\pi_1$ vanish,
		\item half-flat if $\pi_0=\pi_1=\nu_1=\pi_2=0$. \label{halfflatdefinition}
	\end{enumerate}
\end{Def}
\noindent The next notion that we will need is that of a $G_2$-structure which we now define.
\subsection{$G_2$-structures}
\begin{Def}\label{g2definition}
	A $G_2$-structure on a $7$-manifold $L^7$ is given by a $3$-form $\vp$ that can be identified at each point $p\in L^7$ with the standard one on $\R^7$:
	\begin{equation*}
	\vp_0=dx_{123}+dx_{145}+dx_{167}+dx_{246}-dx_{257}-dx_{347}-dx_{356}, \end{equation*}
	where $x_1,\dots,x_7$ denote the coordinates on $\R^7$ and $dx_{ijk}$ is shorthand for $dx_i \w dx_j \w dx_k$.
\end{Def}
The reason for this nomenclature stems from the fact that the subgroup of $GL(7,\R)$ which stabilises $\vp_0$ is isomorphic to the Lie group $G_2$. Since $G_2$ is a subgroup of $SO(7)$ \cite{Bryant1987, Salamon1989} it follows that $\vp$ determines (up to a choice of  a constant factor) a Riemannian metric $g_\vp$ and volume form $\vol_\vp$ on $L^7$. Explicitly we define these by
\begin{equation}
\frac{1}{6}\ (U\ip \vp)\w (V\ip \vp) \w \vp = g_\vp(U,V) \vol_\vp, \label{metricg2} \end{equation}
where $U,V$ are vector fields on $L^7$. In particular, $\vp$ defines a Hodge star operator $*_\vp.$ The set of $3$-forms defining a $G_2$-structure is an open set in the space of sections of $\Lm^3$ which we denote by $\Lm^3_+(L^7)$. In this article we will only consider the situation when $\vp$ is closed. In this case we have
\begin{equation}
	d*_\vp\vp=\tau \w \vp,\label{g2torsion}
\end{equation}
for a unique $2$-form $\tau \in \mathfrak{g}_2 \cong \Lm^2_{14}\hookrightarrow \Lm^2\cong \mathfrak{so}(7)$ where 
\begin{align}
	\Lm^2_{14} &=\{ \al \in\Lm^2 \ | \ \al\w*_\vp\vp=0 \}\label{g2description2214}\\
	&=\{ \al \in\Lm^2 \ | \ *_\vp(\al\w \vp)=- \al \}.\nonumber
\end{align}
The $2$-form $\tau$ is called the intrinsic torsion of the closed $G_2$-structure. On $(L^7,\vp)$ there is also a natural $G_2$ equivariant map given by
\begin{align*}
j: \Lm^3 &\to S^2\nonumber\\
j(\gamma)(U,V) &=*_\vp( (U \ip \vp) \w (V \ip\vp) \w \gamma ).\label{isosym}
\end{align*}
The kernel of $j$ is 
\begin{gather*}
\Lm^3_7 =\{ U\ip *_\vp\vp \ | \ U\in \Gamma(TL) \}
\end{gather*}
and as $G_2$ modules we have
\begin{equation*}\Lm^3= \langle\vp \rangle \oplus \Lm^3_7\oplus \Lm^3_{27} \end{equation*}
where $\Lm^3_{27}$ is identified with the space of traceless symmetric $2$-tensors via $j$ while $j(\vp)=6g_\vp$.
If $\vp$ is both closed and coclosed i.e. $\tau=0$, then the holonomy of $g_\vp$ is contained in $G_2$ and $\Ric(g_\vp)=0$. We refer the reader to the classical references \cite{Bryant1987, Salamon1989} for proofs of the aforementioned facts.
The problem of constructing examples with holonomy \textit{equal} to $G_2$ is very hard. The Laplacian flow was introduced as a means to tackle this problem. 
\subsection{The Laplacian flow}
The Laplacian flow (\ref{laplacianequa})-(\ref{laplacianequa3}) preserves the closed condition i.e. $d\vp(t)=0$ for all $t$ and thus (\ref{laplacianequa}) is equivalent to 
\begin{equation}
\ddt \vp(t) =d\tau(t) .\label{LF}
\end{equation} 
In the compact setting, Hitchin gives the following interpretation of the flow. Consider the functional $\Psi : [\vp_0]^{+} \to \R^{+}$ defined by
\[\Psi(\rho):= \frac{1}{7} \int_L \rho \w *_{\rho} \rho = \int_L \vol_{\rho}\]
where $[\vp_0]^+:=\{\vp_0+d\beta\in \Lm^3_+(L) \ |\ \beta \in \Lm^2(L)\}$ denotes an open set in $[\vp_0]\in H^3(L,\R).$ Then computing the Euler-Langrange equation Hitchin finds that the critical points of $\Psi$ satisfy $d*_\rho \rho=0$. The Laplacian flow is the gradient flow of $\Psi$ with respect to an $L^2$ norm induced by $g_{\rho}$ on $[\vp_0]^+.$ The Hessian of $\Psi$ at a critical point is non-degenerate transverse to the action of the diffeomorphism group and in fact is negative definite. Thus $\Psi$ can be interpreted as a Morse-Bott functional and the torsion free $G_2$-structures correspond to the local maxima. In the non-compact setting this interpretation is not valid but the Laplacian flow is still well-defined and critical points are still torsion free $G_2$-structures cf. \cite{Lotay2017}. 

In \cite{Bryant06someremarks}, Bryant computes the evolution equations for the following geometric quantities under the Laplacian flow
\begin{gather}
\ddt(*_\vp\vp)= \frac{1}{3}\|\tau\|^2_{\vp}*_\vp\vp-*_\vp d\tau,\label{dualevolution} \\
\ddt(g_\vp )=-2 \Ric(g_\vp)+\frac{1}{6}\|\tau \|^2_\vp g_\vp +\frac{1}{4}j(*_\vp(\tau \w \tau)) ,\label{metricevolution}\\
\ddt(\vol_\vp)=\frac{1}{3}\|\tau \|^2_\vp \vol_\vp. \label{volumeevolution}
\end{gather}
We see immediately that up to lower order terms (\ref{metricevolution}) coincides with the Ricci flow of $g_\vp$. In \cite[(4.30)]{Bryant06someremarks} Bryant also derives an expression for the Ricci tensor in terms of the torsion form only and thus, one can express (\ref{metricevolution}) only in terms of the torsion form as
\begin{equation}
\ddt(g_\vp)=-\frac{1}{3}\|\tau\|^2_\vp g_\vp+\frac{1}{2}j(d\tau).\label{metricevolution2}
\end{equation}
The simplest solutions to (\ref{laplacianequa}) are those that evolve by the symmetry of the flow. If $\vp_0$ satisfies 
\begin{equation}\label{solitonequationoriginal}
\mathrm{\Delta}_{\vp_0}\vp_0 = \lambda\cdot \vp_0 + \mathcal{L}_V \vp_0
\end{equation} 
for a vector field $V$ and constant $\lambda$ then \begin{equation*}
\vp_t:=(1+\frac{2}{3}\lambda t )^{\frac{3}{2}} F^{*}_t \vp_0 ,
\end{equation*} 
where $F_t$ is the diffeomorphism group generated by $U(t)= (1+\frac{2}{3} \lambda t)^{-\frac{2}{3}}V$, is a solution to flow and $\vp_0$ is called a Laplacian soliton \cite{Lotay2017}. Depending on whether $\lambda$ is positive, zero or negative the soliton is called expanding, steady or shrinking respectively. If $V$ is a gradient vector field then we call them gradient solitons. 
We refer the reader to \cite{Lotay2017} for foundational theory for the Laplacian flow.\\

\noindent\textbf{Notations.} In what follows we shall write $\Lm^i$ for the space of differential forms, as well as their sections, omitting reference to the underlying space $P^6$ or $L^7$ unless there is any possible ambiguity. For a $k$-form $\al$ we will denote by $(\al)^k_l$ its projection in $\Lm^k_l$.
We shall also omit pullback signs and identify forms on $P^6$ with their pullbacks on $L^7$.

\section{$S^1$ reduction of closed $G_2$-structure}\label{quotientofclosedg2structures}

Our aim in this section is to characterise $S^1$-invariant closed $G_2$-structures purely in terms of the data on the quotient space. In particular, we derive a Gibbons-Hawking Ansatz type construction for closed $G_2$-structures, see Theorem \ref{GHcalibratedG2}.

Let $\vp$ be a closed $G_2$-structure on $L^7$ which is invariant under a free $S^1$ action generated by a vector field $Y$. We define a connection $1$-form $\xi$ on $L^7$ by 
\begin{equation} \xi(\cdot):= H^2 g_\vp(Y,\cdot),\label{connection}\end{equation}
where $H:=\|Y\|^{-1}_\vp$. 
The hypothesis that this circle action is free i.e. $Y$ is nowhere vanishing ensures that $H$ is well-defined and moreover, by definition $H$ is also $S^1$-invariant. 
The quotient space $P^6:=L^7/S^1$ inherits an $SU(3)$-structure $(\om,\Om)$ given by
\begin{gather}
	\vp = \xi \w \om + H^{3/2} \Om^+,\label{g2form}\\
	*_\vp\vp = \frac{1}{2} H^2 \om^2-\xi \w H^{1/2}\Om^- ,\\
	g_\vp=H^{-2}\xi^2+Hg_\om.
\end{gather}
The curvature form $d\xi$ is invariant under $Y$ and horizontal i.e. 
\begin{equation}
	Y \ip d\xi = \mathcal{L}_Y \xi=0,
\end{equation}
and as such it descends  to $P^6$.
Since $\mathcal{L}_Y\vp=0$ and $\vp$ is closed, we have that
\begin{gather}
	d \om = 0 ,\\
	d \Om^+ = -\frac{3}{2}H^{-1} dH \w \Om^+ - H^{-\frac{3}{2}}d \xi \w \om.\label{torsu31}
\end{gather}
From (\ref{compatibilitysu3}) and the fact that $\vp \w \om$ is closed it follows that
\begin{equation}
	d\xi\w\om^2=0.\label{curvaturerestriction}
\end{equation}
We can also express the torsion form $\tau$ as 
\begin{equation}\tau= \tau_h + \xi \w \tau_v\end{equation} 
for a $2$-form $\tau_h$ and a $1$-form $\tau_v$ which are both basic i.e. they are horizontal and $S^1$-invariant (since  ${\mathcal{L}}_Y \tau=0$). Thus, $\tau_h$ and $\tau_v$ are really pullback of forms on $P^6$. Our goal is to encode the intrinsic torsion $\tau$ of the $S^1$-invariant closed $G_2$-structure (and its derivatives) only in terms of the data $(P^6,\om,\Om,H)$.

As $\tau \in \Lm^2_{14} $ it follows from (\ref{g2description2214}) that
\begin{gather}
	\tau_h \wedge \omega^2=0 \label{noomega},\\
	\tau_v \w \frac{1}{2}H^{3/2} \om^2=\tau_h\w \Om^-.\label{relation1}
\end{gather}
From (\ref{decomposition2form}), we note that (\ref{curvaturerestriction}) and (\ref{noomega}) imply that $d\xi$ and $\tau_h$ have no $\om$-component i.e. 
\begin{equation}
d\xi=(d\xi)^2_6+(d\xi)^2_8 \text{\ \ \ and\ \ \ }\tau_h=\tau_6+\tau_8\in \Lm^2_6\oplus\Lm^2_8 
\end{equation} 
The latter is not surprising since $\dim(\mathfrak{g}_2)=6+8$.
In terms of the $SU(3)$-structure we can express the condition $d*_\vp\vp=\tau\w \vp$ as
\begin{gather}
	d \Om^-= H^{-\frac{1}{2}}\tau_6 \w \om + (H \tau_v -\frac{1}{2}H^{-1}d^c H) \w \Om^++H^{-\frac{1}{2}}\tau_8 \w \om,\label{torsu32}\\
	H dH \w \om^2 - (d\xi)^2_6 \w H^{1/2}\Om^-=\tau_6 \w H^{\frac{3}{2}}\Om^+,\label{relation2}
\end{gather}
where we recall $d^cH=JdH.$
Observe that the forms $(d\xi)^2_6$, $\tau_v$ and $\tau_6$ are related by (\ref{relation1}), (\ref{relation2}) and the fact that $\pi_1$ appears in both (\ref{su3torsion2}) and (\ref{su3torsion3}). Thus, these forms are all essentially equivalent, to be more precise we have:
\begin{Lemma}\label{nicelemma}
	The torsion forms $\tau_v$ and $\tau_6$ are determined by the curvature component $(d\xi)^2_6$ via 
	$$\tau_v =-2 H^{-2}(d^cH + J(\gamma^1_6)) \ \ \ \text{\ and\ } \ \ -2\tau_6 = H^{\frac{3}{2}}*_\om(\tau_v \w \Om^+),$$
	where the $1$-form $\gamma^1_6$ is defined by $H^{-\frac{1}{2}}(d \xi)^2_6 \w \om = \gamma^1_6 \w \Om^+$.
\end{Lemma}
\begin{proof}
	Let $\tau_6\w \om = H^{\frac{3}{2}}\beta_6\w \Om^+$ for a $1$-form $\beta_6$ (note that the existence of $1$-forms $\beta_6$ and $\gamma_6^1$ follows from (\ref{lambda262})), then using Lemma \ref{interchangelemma} we can express the $SU(3)$ torsion forms from (\ref{torsu31}) and (\ref{torsu32}) in irreducible summands as
	\begin{align}
		d \Om^+ = (-\frac{3}{2}H^{-1}dH - H^{-1}\gamma^1_6) \w \Om^+-H^{-\frac{3}{2}} (d \xi)^2_8\w \om,\label{omplustorsion}\\
		d \Om^- = (H(\tau_v + \beta_6)-\frac{1}{2}H^{-1}d^c H)\w \Om^++H^{-\frac{1}{2}} \tau_8 \w \om,\label{omminustorsion}
	\end{align}
	and hence comparing the torsion $1$-form $\pi$ from (\ref{su3torsion2}) and (\ref{su3torsion3}) we have that
	\[d^c H + J \gamma^1_6 = - H^2(\tau_v + \beta_6).\]
	From (\ref{relation1}) and using $(2)$ of Lemma \ref{interchangelemma} we find that
	\[2H^{-3/2}\tau_6 \w \om=2H^{-3/2}*_\om\tau_6=J(\tau_v)\w \Om^-.\]
	Another application of Lemma \ref{interchangelemma} now shows that 
	\[2\beta_6 \w \Om^+=2H^{-3/2}\tau_6 \w \om=J(\tau_v)\w \Om^-=-\tau_v \w \Om^+\]
	i.e. $\tau_v=-2\beta_6$ and this completes the proof.
\end{proof}
In particular Lemma \ref{nicelemma} asserts that the torsion form $\tau_h$ completely determines $\tau_v$.
We can now express the $SU(3)$ torsion forms as:
\begin{gather}
	\tau_6= H^{-\frac{1}{2}} *_\om((d^c H+J\gamma_6^1) \w \Om^+), \\
	\tau_8 = -H^{\frac{1}{2}}*_\om d\Om^- - *_\om(H^{-\frac{1}{2}} (\frac{3}{2} d^c H+J\gamma^1_6) \w \Om^+). \end{gather}
Since $\tau_h=\tau_6+\tau_8$, it follows, using Lemma \ref{interchangelemma}, that
\begin{align}
	\tau_h &= -H^{1/2}*_\om d*_\om\Om^+-*_\om (d^c(H^{1/2})\w\Om^+)\nonumber\\
	&= d^{*_\om}(H^{1/2}\Om^+),\label{tauhdef}
\end{align} 
where $d^{*_\om}:=-*_\om d *_\om$ denotes the codifferential on $P^6$
and hence the $G_2$ torsion form can be expressed as:
\begin{align}
\tau 
&=  d^{*_\om}(H^{1/2}\Om^+) -2 H^{-2} \xi \w (d^c H + J\gamma^1_6).\label{torsionformS1invt}
\end{align}
In particular we have $d^{*_\om}\tau_h=0,$ which is equivalent to
\begin{equation}d\tau_6 \w \om = d\tau_8 \w \om. \label{t68}
\end{equation}
i.e. $(d\tau_6)^3_6=(d\tau_8)^3_6$.
One can also extract the $\Om^-$-component of $d\tau_h\in \Lm^3$ from the next Proposition. 
\begin{Prop}
The following holds
\begin{equation}
d(H^{3/2}\tau_h \w \Om^+)= g_{\om}(d\xi, \tau_h)\vol_\om, \label{div}
\end{equation}
or equivalently $(d\tau_h)\w H^{3/2}\Om^+=2 g_\om(d\xi, \tau_6)\vol_\om$.
In particular, if $P^6$ is compact then
\[\int_{P^6} g_{\om}(d\xi,d^{*_{\om}}(H^{1/2}\Om^+)) \vol_\om=0.\]
\end{Prop}
\begin{proof}
Differentiating (\ref{g2torsion}) we have
\[d(H^{3/2}\tau_h \w \Om^+)+d\xi \w \tau_h \w \om +d\xi \w \tau_v \w H^{3/2 }\Om^+=0\]
and
\[d(\tau_h \w \om + \tau_v \w H^{3/2}\Om^+)=0.\]
One can show that the latter equation is just the condition that $d^{*_\om}\tau_h=0$ while using Lemma \ref{nicelemma} the former simplifies to (\ref{div}). The equivalence is easily deduced since $d\vp=0$ implies that $d\xi \w \om+d(H^{3/2}\Om^+)=0.$
\end{proof}
Having now expressed $\tau$ only in terms of the data $(\xi,\om, \Om, H)$, we need to show that we can recover $(L^7,\vp)$ from $(P^6,\om,\Om,H)$, which is achieved by the next Theorem.
\begin{Th}[Gibbons-Hawking Ansatz for closed $G_2$-structures]\label{GHcalibratedG2}
Given a 
symplectic manifold $(P^6,\om)$ admitting an $SU(3)$-structure $(\om,\Om)$ and a positive function $H:P^6\to \R^+$ satisfying 
\begin{equation}\label{integralconditiong2}
[-*_\om\big(*_\om (d(H^{3/2}\Om^+))\w\om\big)]\in H^2(P^6,2\pi\mathbb{Z}),
\end{equation}
then $\vp:=\xi\w \om + H^{3/2}\Om^+$ defines a closed $G_2$-structure on the total space of the $S^1$ bundle determined by (\ref{integralconditiong2}) and the curvature of the connection form $\xi$ is given by
\begin{equation}
d\xi= -*_\om(d^{*_\om}(H^{3/2}\Om^-)\w\om).\label{harmonicconditiong2}
\end{equation}
\end{Th}
\begin{proof}
In view of the above quotient construction we only need to prove that $\xi$, defined by (\ref{connection}), satisfies (\ref{harmonicconditiong2}). Since $d\xi$ defines an \textit{integral} cohomology class on $P^6$, the result follows from Chern-Weil theory, or more precisely \cite[Theorem 2]{Kobayashi1956}.

Applying $*_\om$ to (\ref{torsu31}), and using (\ref{lambda26}) and (\ref{lambda28}), we find
\[(d\xi)^2_8=(d\xi)^2_6+*_\om(d(H^{3/2}\Om^+))\]
Consider now the automorphism $\mathbf{L}:\Lm^2_6\oplus \Lm^2_8 \to \Lm^2_6\oplus \Lm^2_8$ given by
\[\mathbf{L}(\al) =*_\om(\al \w \om)\]
which acts as the identity on $\Lm^2_6$ and minus identity on $\Lm^2_8$. Then 
\[d\xi=-\mathbf{L}(*_\om(d(H^{3/2}\Om^+)))=-*_\om(d^{*_\om}(H^{3/2}\Om^-)\w\om),\]
using that $\Om^+=-*_\om\Om^-$ and this completes the proof.
\end{proof}
\begin{Rem}
Let $B$ be an open set of $\R^3$ endowed with the Euclidean metric and volume form, and hence Hodge star operator $*_0$. The Gibbons-Hawking Ansatz states that given a positive  harmonic function $h:B \to \R^+$  such that 
\begin{equation}
[-*_{0}dh]\in H^2(B,2\pi\mathbb{Z})\label{integralityGH}
\end{equation}
then
\[g_M=h^{-1}\theta^2+h g_B\]
defines a hyperK\"ahler metric on the total space $M^4\to B$ of the $S^1$ bundle determined by (\ref{integralityGH}), where $\theta$ is a connection $1$-form satisfying $d\theta =-*_{0}dh.$ 
In Theorem \ref{GHcalibratedG2} condition (\ref{integralconditiong2}) is the higher dimensional analogue of the `integrality' condition (\ref{integralityGH}) that figures in the Gibbons-Hawking Ansatz and the (linear) harmonic condition $d^{*_0}dh=0$ on $h$ is replaced by the (non-linear) condition
\begin{equation}
d^{*_\om}(d^{*_\om}(H^{3/2}\Om^-)\w\om)=0
\end{equation}
on the pair $(H,\Om)$. So the data $(\om,\Om^-,H)$ is sufficient to recover the $G_2$-structure $\vp$ since the curvature of $\xi$ is already determined by (\ref{harmonicconditiong2}).
\end{Rem}
We can also characterise the torsion free $G_2$-structures in terms of the data on the base $P^6$ by the following Proposition.
\begin{Prop}
Assuming we are in the situation of Theorem \ref{GHcalibratedG2}, then
	\[ d*_\vp\vp=0 \textit{\ \ \  if and only if } \ \  d{*_\om}(H^{\frac{1}{2}} \Om^+)=0 .\]
	If this holds, then $(P^6,\om,\Om)$ is a Calabi-Yau $3$-fold if and only if $H$ is constant.
\end{Prop}
\begin{proof}
The first part follows from the equivalence between the conditions $d*_\vp\vp=0$ and $\tau=0$ and the fact that $d^{*_\om}(H^{1/2}\Om^+)=\tau_h=0$ implies (from Lemma \ref{nicelemma}) that $\tau_v=0$.
If furthermore $H$ is constant then from Lemma \ref{nicelemma} we have that $(d\xi)^2_6=0$ i.e. $d\xi=(d\xi)^2_8$. Differentiating the relation
\[d\xi \w \Om^+=0\]
from (\ref{lambda282}) and using (\ref{torsu31}) then shows that $\|d\xi\|_{\om}=0$ i.e. $L^7=S^1 \times P^6$. To complete the proof we need to show that if $(P^6,\om,\Om^+)$ is Calabi-Yau then $H$ is constant which follows immediately from
\[ 0=d{*_\om}(H^{\frac{1}{2}} \Om^+)=\frac{1}{2}H^{-1/2}dH\w \Om^-.\]
\end{proof}
Having encoded the data of a closed $S^1$-invariant $G_2$-structure in terms of the data on the quotient we derive the evolution equations for the data $(\om,\Om,H, \xi)$ under the Laplacian flow in the next section. 

\section{The $S^1$-invariant Laplacian flow}\label{LFsection}
\subsection{$S^1$-invariant flow equations}
Consider the Laplacian flow starting from an $S^1$-invariant closed $G_2$-structure. Then by the existence and uniqueness of the flow, at least in the compact case, it follows that this symmetry persists i.e the solution to the flow can be expressed as (\ref{g2form}) for as long as it exists cf. \cite[Corollary $6.7$]{Lotay2017}. Note also that since $\mathcal{L}_Y \ddt(\xi)=\ddt(\mathcal{L}_Y \xi)=0$ and $Y \ip \ddt(\xi)=0$ it follows that $\ddt(\xi)$ is a basic $1$-form i.e. it corresponds to a $1$-form on $P^6$. Thus, in the $S^1$-invariant setting, from (\ref{torsionformS1invt})  we see that (\ref{LF}) becomes equivalent to the following evolution equations on $(\om,\Om,H,\xi)$: 
\begin{gather}
	\ddt (\om) = -2dd^c(H^{-1})+2d(H^{-2} J\gamma^1_6), \label{omegaevolution}\\
	\ddt (\xi) \w \om + \ddt(H^{3/2}\Om^+)= -d*_\om d(H^{\frac{1}{2}}\Om^-)+2d\xi \w (d^c(H^{-1}) - H^{-2} J\gamma^1_6). \label{mixedevolution}
\end{gather}
We omit writing the dependence on $t$ to ease the notation.
\begin{Rem}
Observe that (\ref{omegaevolution}) agrees with the fact that since $\vp$ remains in its cohomology class so does $\om$. 
\end{Rem} 
\noindent The main result of this section can be summed up as follows:
\begin{Th}
Given a compact symplectic $SU(3)$-structure $(P^6,\om_0,\Om_0)$ together with the data of an $S^1$ bundle with connection $1$-form $\xi_0$ and positive function $H_0$ as in Theorem \ref{GHcalibratedG2}, then the coupled flow defined by
\begin{gather}
\ddt(\om)=-d\tau_v\label{omflow}\\
\ddt(H^{1/2}\Om^-)=\frac{1}{3}(H^{-3/2}\|\tau_h\|^2_\om+H^{3/2}\|\tau_v\|^2_\om)\Om^--H^{-1}*_\om(d\tau_h+d\xi \w \tau_v)\\
\ddt(\log(H))=\frac{1}{6}(H^{-2}\|\tau_h\|^2_\om+H\|\tau_v\|^2_\om)+\frac{1}{2}g_\om(d\tau_v,\om)\label{Hflow}\\
\ddt(\xi)=-*_\om((J\circ d\tau_v)\w (H^{3/2}\Om^-))\label{xiflow}
\end{gather}
where $\tau_h:=d^{*_\om}(H^{1/2}\Om^+)$ and $\tau_v:=2 d^c(H^{-1})-2H^{-2}J\gamma_6^1$ with $\gamma_6^1$ defined by $ \gamma^1_6 \w \Om^+=H^{-\frac{1}{2}}(d \xi)^2_6 \w \om $, admits short time existence and uniqueness for the initial data $(\om_0,\Om_0,\xi_0,H_0)$. Moreover, $(P^6,\om)$ stays symplectic and $d(H^{3/2}\Om^+)=-d\xi \w \om$ holds for as long as the flow exists.
\end{Th}
\begin{proof}
To complete the proof we just need to show that (\ref{omflow})-(\ref{xiflow}) correspond to the  Laplacian flow, which is the content of the rest of this section. The existence and uniqueness of the flow is then immediate from that of the Laplacian flow \cite{Lotay2017}. Since $d\vp=0$ as long as the flow exists it follows that $d\om=0$ and $d(H^{3/2}\Om^+)=-d\xi \w \om$.
\end{proof}
Note that in view of Theorem \ref{GHcalibratedG2} we already know that the evolution equation for $\xi$ is a consequence of (\ref{omflow})-(\ref{Hflow}). This can also be seen by inspecting expression (\ref{xiflow}). Before deriving the evolution equations for the data $(\om,\Om,\xi,H,\vol_\om)$ on $P^6$ we first give expressions for quantities that will appear in the evolution equations.
\begin{Lemma}\label{normlemma}The norms of the torsion forms can be expressed as
	\begin{enumerate}
		\item $ \|\gamma_6^1\|^2_\om=\frac{1}{2}H^{-1}\|(d\xi)^2_6 \|^2_\om  $
		\item $
		\|\tau \|^2_\vp=H^{-2}(\|\tau_8 \|^2_\om + 3\|\tau_6 \|^2_\om) =H^{-2}\|\tau_h\|^2_\om+H\|\tau_v\|^2_\om$
		\item $\|\tau_6\|^2_\om=2H^{-1}\|dH+\gamma^1_6\|^2_\om$
		\item $\|d\Om^- \|^2_\om=H^{-1}\|\tau_8 \|^2_\om+H^{-2}(\frac{9}{2}\|dH \|^2_\om+2\|\gamma_6^1 \|^2_\om +6 g_\om(dH,\gamma_6^1) ) $
	\end{enumerate}
\end{Lemma}
\begin{proof}
	The proof is a direct calculation using the expressions from the previous section. We prove ($1$) as an example:
	\[\gamma^1_6\w*_\om\gamma^1_6=\frac{1}{2}H^{-1/2}\gamma^1_6\w(d\xi)^2_6\w\Om^+=\frac{1}{2}H^{-1}(d\xi)^2_6\w *_\om(d\xi)^2_6,\]
	where the first equality follows from ($2$) of Lemma \ref{interchangelemma} and the definition of $\gamma^1_6$. The second equality is again just by the definition of $\gamma^1_6$. The proofs for the rest follow by similar computations. 
\end{proof}

\begin{Prop}\label{PropEvolutionofxi}The evolution equation for the connection form $\xi$ is given by
	\begin{align}\label{evolutionofxi}
		\ddt (\xi)  = -2H^{3/2}*_\om(J(dd^c(H^{-1})-d(H^{-2}J\gamma_6^1))\w \Om^-)
	\end{align}
\end{Prop}
\begin{proof}
	From (\ref{mixedevolution}) we have
	\[\ddt(\xi)\w \om^2 = (d\tau_h+d\xi \w \tau_v)\w \om-\ddt(H^{3/2}\Om^+)\w \om.\]
	Taking the time derivative of the relation $H^{\frac{3}{2}}\Om^+\w \om=0$ and using (\ref{omegaevolution}) we get
	\[\ddt(\xi)\w \om^2 = (d\tau_h+d\xi \w \tau_v)\w \om-d\tau_v\w (H^{3/2}\Om^+).\]
Since $d\xi \w \om=-d(H^{3/2}\Om^+)$, we can rewrite the above as
	\begin{equation}
	\ddt(\xi)\w \om^2 = d(\tau_h\w \om + \tau_v\w H^{3/2}\Om^+)-2 d\tau_v\w (H^{3/2}\Om^+).\label{xievol}
   \end{equation}
From Lemma \ref{nicelemma} and (\ref{t68}) we see that 
\[
d(\tau_h\w \om + \tau_v\w H^{3/2}\Om^+)=d(\tau_8-\tau_6)\w \om=0.
\]
Finally using $(2)$ of Lemma \ref{interchangelemma} we can rewrite (\ref{xievol}) as (\ref{evolutionofxi}).
\end{proof}
\begin{Prop}
	\begin{align}
		\ddt(\vol_\om)&=-d\big((d^c(H^{-1})-H^{-2}J(\gamma_6^1)) \w \om ^2\big)\\
		&=-2g_{\om}(dd^c(H^{-1})-d(H^{-2}J(\gamma_6^1)),\om)\vol_\om
	\end{align}
\end{Prop}
\begin{proof}
	This follows directly from (\ref{omegaevolution}) and the fact that $\vol_\om=\frac{1}{6}\om^3. $
\end{proof}
\begin{Prop}\label{PropEvolutionofH}
\begin{align}\label{evolutionofH}
\ddt(H)=&-H^{-1}d^{*_{\om}}d(H)-2H^{-2}g_{\om}(dH,\gamma_6^1)-H^{-2}\|dH \|^2_\om\\
&+\frac{1}{6}H^{-1}\|\tau_8\|^2_\om+\frac{1}{2}H^{-3}\|(d\xi)^2_8\|^2_\om\nonumber
\end{align}
We can also expressed the above more compactly as
\begin{equation}\label{evolutionofHcompact}
\ddt(\log(H))=\frac{1}{6}\|\tau\|^2_\vp+g_{\om}(dd^c(H^{-1})-d(H^{-2}J(\gamma_6^1)),\om).\end{equation}
\end{Prop}
\begin{proof}
Since $H=g_{\vp}(Y,Y)^{-1/2}$ we can use (\ref{metricevolution2}) to write down its evolution equation: 
\begin{equation}\ddt(H)=\frac{1}{6}H\|\tau\|^2_\vp-\frac{1}{4}H^3j(d\tau)(Y,Y).\label{evolutionofHproof}\end{equation}
From Lemma \ref{normlemma} we can express $\|\tau\|^2_\vp$ in terms of the torsion of the $SU(3)$-structure as
\begin{equation}
\|\tau\|^2_\vp=H^{-2}\|\tau_8\|^2_\om+6H^{-3}\|dH+\gamma_6^1\|^2_\om.\end{equation}
Thus, we only need to simplify the term 
\begin{equation}\frac{1}{2}j(d\tau)(Y,Y)=*_\vp(\om\w\om\w \xi\w d(H^{-2}d^cH+H^{-2}J\gamma^1_6)) \label{jdtau}\end{equation} 
which is straightforward to do, except for the term involving $d(J\gamma^1_6)$. Using Lemma \ref{interchangelemma} and (\ref{omplustorsion}) we compute 
\begin{align*}
d(H^{1/2}J\gamma_6^1\w\om\w\om)&=-d((d\xi)^2_6 \w\Om^+ )\\
&=-d\xi \w d\Om^+\\
&= J(\gamma_6^1)\w \om^2 \w (\frac{3}{2}H^{-1/2}dH+H^{-1/2}\gamma_6^1)-H^{-3/2}\|(d\xi)^2_8\|_\om^2\\
&=(3H^{-1/2}g_\om(dH,\gamma_6^1)+2H^{-1/2}\|\gamma_6^1\|^2_\om-H^{-3/2}\|(d\xi)^2_8\|^2_\om)\vol_\om.
\end{align*}
We also have that
\[dd^c H \w \om^2=d(d^cH\w \om^2)=-2d(*_\om dH)\]
and
\[dH\w d^cH \w \om^2=-2\|dH\|^2_\om \vol_\om.\]
Using the last three expressions and the fact that $*_\vp(\xi \w \vol_\om)=H^{-2}$, one can rewrite (\ref{jdtau}) in terms of data on $P^6$. Substituting all this in (\ref{evolutionofHproof}) completes the proof for the first equation. 
Rather than unwinding the various relations between the torsion forms given in the previous section, one can prove the second expression more directly using (\ref{volumeevolution}) together with the fact that $\vol_\vp=H^{2}\xi\w \vol_\om$ and the evolution equation for $\vol_\om$.
\end{proof}
\begin{Rem}
Observe that even if $H$ is initially constant, i.e. the $S^1$ orbits have constant size, this is not generally preserved in time. Indeed from (\ref{evolutionofH}) we see that if $dH_0=0$ then $\partial_t(H)\big|_{t=0}\geq 0$, so the size of the $S^1$ orbit is expected to shrink initially.
\end{Rem}
\begin{Prop}The evolution equation for $\Om^+$ is given by
\begin{align}	
\ddt(\Om^+)=&-\Big(\frac{1}{4}\|\tau\|^2_\vp+\frac{3}{2}g_{\om}(dd^c (H^{-1})-d(H^{-2}J(\gamma_6^1)),\om)\Big)\Om^+\\ &+ H^{-3/2}\Big(dd^{*_\om}(H^{1/2}\Om^+)+2d\xi \w (d^c (H^{-1})-H^{-2}J(\gamma_6^1))\Big)\nonumber\\
&+2*_\om\Big(\big(J(dd^c(H^{-1})-d(H^{-2}J\gamma_6^1))\w \Om^-\big)\Big)\w \om.\nonumber
\end{align}
\end{Prop}
\begin{proof}
From (\ref{mixedevolution}) we have
\begin{align*}
	  \ddt(\Om^+)=&-\frac{3}{2}\ddt(\log(H))\cdot \Om^+ \\
	  &-H^{-3/2}\big(d*_\om d(H^{\frac{1}{2}}\Om^-)-2d\xi \w (d^c(H^{-1}) - H^{-2} J\gamma^1_6)\big)\\
	  &-H^{-3/2}\ddt (\xi) \w \om. 
\end{align*}
From Propositions \ref{PropEvolutionofxi} and \ref{PropEvolutionofH} we already know the evolution equations for $\xi$ and $H$. The result follows immediately by substituting in (\ref{evolutionofxi}) and (\ref{evolutionofHcompact}).
\end{proof}
\begin{Prop}The evolution equation for $\Om^-$ is given by
\begin{align}
\ddt(\Om^-) =\ &\frac{1}{4}\bigg( H^{-2}\|\tau_h\|^2_\om+H\|\tau_v\|^2_\om - g_\om\big(2dd^c(H^{-1})-2d(H^{-2} J\gamma^1_6), \om\big)  \bigg)\Om^-\nonumber\\ 
&- H^{-3/2}\Big(d^{*_\om}d(H^{1/2}\Om^-) 
+*_\om\big(d\xi \w (2 d^c(H^{-1})-2H^{-2}J\gamma_6^1)\big)\Big) 
\end{align}
\end{Prop}
\begin{proof}
Comparing terms in (\ref{dualevolution}) involving only $\xi$ i.e. contracting both sides of (\ref{dualevolution}) with $Y$, one has 
\[\ddt(H^{1/2}\Om^-)=\frac{1}{3}\|\tau\|^2_\vp H^{1/2}\Om^--H^{-1}*_\om(d\tau_h+d\xi \w \tau_v).\]
From the latter and using (\ref{evolutionofHcompact}) we compute
\begin{align*}
\ddt(\Om^-) =\ &\Big(\frac{1}{3}\|\tau\|^2_\vp-\frac{1}{2}\ddt(\log H)\Big) \Om^--H^{-3/2}*_\om(d\tau_h+d\xi \w \tau_v)\\
=\ &\Big(\frac{1}{4}\|\tau\|^2_\vp-\frac{1}{2}g_{\om}(dd^c(H^{-1})-d(H^{-2}J(\gamma_6^1)),\om)\Big) \Om^-\\ 
&-H^{-3/2}*_\om(d\tau_h+d\xi \w \tau_v).
\end{align*}
The result now follows from $(2)$ of Lemma \ref{normlemma} and using the definitions of $\tau_h$ and $\tau_v$.
\end{proof}
\begin{Prop}
\begin{align}
\ddt(g_\om)=&-\bigg(\frac{1}{2}\|\tau \|^2_\vp+\frac{1}{6}g_\om(d \tau_v,\om)-2H^{-3/2}(d\xi \w \tau_v + d\tau_h)^{3+}_1 \bigg)g_{\om}\label{metricevolution3}\\
&+\frac{1}{2}H^{-1}j( (d\xi \w \tau_v
+d\tau_h)^3_{12} + \xi \w(d\tau_v)^2_8) .\nonumber
\end{align}
\end{Prop}
\begin{proof}
The idea is to again use the evolution equation for $g_\vp$. Since $\ddt(g_\om)$ only evolves on the base $P^6$ we can ignore terms involving $\xi$ in (\ref{metricevolution2}). 
Thus, we have that 
\begin{equation} \ddt(g_\om)=-\ddt(\log H) g_\om -\frac{1}{3}\|\tau \|^2_\vp g_\om+\frac{1}{2}H^{-1}j(d\tau)\bigg|_{P^6}\label{evolg}
\end{equation}
As $SU(3)$ modules we have the following decomposition
\begin{align*}
\langle\vp \rangle\oplus \Lm^3_{27}(L)\cong S^2(\R^7)&=S^2(\R\oplus \R^6)\\
&=\langle\xi^2 \rangle\oplus (\xi \odot \R^6) \oplus S^2(\R^6)\\
&=\langle\xi^2 \rangle\oplus (\xi \odot \R^6) \oplus\langle g_\om \rangle \oplus S^2_0(\R^6)\\
&\cong\langle\xi^2 \rangle\oplus (\xi \odot \R^6) \oplus \langle g_\om \rangle \oplus \Lm^2_8(P)\oplus \Lm^3_{12}(P)
\end{align*}
By abuse of notation we are identifying the cotangent spaces of $P^6$ and $L^7$ with $\R^6$ and $\R^7$ in the above.
It follows that the only terms in $j(d\tau)$ that contribute to the evolution of $g_\om$ belong to the last $3$ summands. Since we have that
$d\tau = d\xi \w \tau_v + d\tau_h - \xi \w d\tau_v,$
the only terms that can arise in the evolution of $g_\om$ are the $\langle\Om^+ \rangle \oplus \langle\Om^- \rangle \oplus \Lm^3_{12}$ components of $d\xi \w\tau_v+d\tau_h$ which we write as
\begin{gather}
(d\xi \w \tau_v + d\tau_h)^{3+}_{1}\Om^++
(d\xi \w \tau_v + d\tau_h)^{3-}_{1}\Om^-+
(d\xi \w \tau_v + d\tau_h)^3_{12}
\end{gather}
	and the $\langle \om \rangle\oplus \Lm^2_8$ components of
	$d\tau_v=H^{-2}*_\om(*_\vp(\xi \w d\tau_v))$ which we can write as
	\begin{equation}
		\frac{1}{3}g_\om(d\tau_v,\om)\om+(d\tau_v)^2_8.
	\end{equation}
	A direct computation (in a $G_2$ coframe) shows that 
	$$j(H^{3/2}\Om^+)=4Hg_\om,$$
	$$j(H^{3/2}\Om^-)=0,$$
	$$j(\xi \w \om)=6H^{-2}\xi^2+2 H g_\om .$$
	Since the orthonormal symmetric tensors $j(\vp)=6g_\vp$ and $j(\xi \w \om-\frac{3}{4}H^{3/2}\Om^+)=6H^{-2}\xi^2-Hg_\om$ span the rank $2$ module $\langle \xi^2,g_\om \rangle$ it follows that as $SU(3)$ modules we have
	$$\Lm^3_{27} \cong \langle 6H^{-2}\xi^2-Hg_\om \rangle \oplus\langle\xi \odot v \rangle_{v\in T^*P} \oplus \Lm^2_8 \oplus \Lm^3_{12}.$$
	We now compute
	\begin{align*}
		j(d\tau)\Big|_{P^6} =&\ j\big((d\xi \w \tau_v
		+d\tau_h)^{3+}_{1}\Om^++(d\xi \w \tau_v
		+d\tau_h)^3_{12}\\ &+ \xi \w(\frac{1}{3}g_\om(d\tau_v,\om)\om+(d\tau_v)^2_8\big)\Big|_{P^6}\\
		=&\ 4H^{-1/2}(d\xi \w \tau_v
		+d\tau_h)^{3+}_{1}g_\om+\frac{2}{3}Hg_\om(d\tau_v,\om)g_\om\\
		&+ j( (d\xi \w \tau_v
		+d\tau_h)^3_{12} + \xi \w(d\tau_v)^2_8).
	\end{align*}
Substituting the latter and (\ref{evolutionofHcompact}) in (\ref{evolg}) gives the result.
\end{proof}
The reader might find the presence of the map $j$ in (\ref{metricevolution3}) rather unusual as the latter is strictly speaking a $G_2$-equivariant map but one can replace it by the corresponding $SU(3)$-equivariant map 
\[\ \iota\oplus\gamma: S^2_0(P)\cong \Lm^2_8(P) \oplus \Lm^3_{12}(P)\]
defined in \cite[Sect. 2.3]{Bedulli2007}. 
 To conclude this section we derive the evolution equations for certain types of differential forms on $(P^6,\om,\Om)$. 
\begin{Lemma}\
\begin{enumerate}
\item Let $\alpha=\alpha^2_6 + \alpha^2_8 \in \Lm^2_6 \oplus \Lm^2_8$  then
\[\ddt(\alpha) \w \om^2 =4 g_\om(dd^c(H^{-1})-d(H^{-2}J\gamma_6^1), \alpha^2_6 - \al^2_8).\]
\item Let $\alpha\in \langle \om \rangle \oplus\Lm^2_8$ then
\[\ddt(\alpha)\w \Om^-=H^{-3/2}\al \w *_\om(d\tau_h+d\xi \w \tau_v).\]
\item Let $\alpha\in \langle \Om^+\rangle \oplus \langle \Om^- \rangle \oplus \Lm^3_{12}$ then
\[\ddt(\alpha)\w \om=2 \al \w  (dd^c(H^{-1})-d(H^{-2}J\gamma_6^1)).\]
\item Let $\alpha\in  \langle \Om^- \rangle \oplus \Lm^3_6 \oplus \Lm^3_{12}$ then
\[\ddt(\alpha)\w \Om^-=H^{-3/2} \al \w  (d\tau_h + d\xi \w \tau_v).\]
\end{enumerate}
\end{Lemma}
\begin{proof}
To prove (1)  we simply differentiate the relation $\alpha \w \om^2=0$ and use (\ref{omegaevolution}). The proofs for the rest are completely analogous. For (2) we differentiate $\al \w \Om^-=0$, for (3) we differentiate $\al \w \om=0$ and for (4) we differentiate $\al \w \Om^-=0$.
\end{proof}
These expressions can be quite useful for extracting the evolution equations for specific components of a given quantity. For instance we can apply (1) and (2) to $\tau_8$ to find the evolution equation for the component of $\ddt(\tau_8)$ in $\langle \om \rangle$ and $\Lm^2_6$ respectively. 
From (\ref{dualevolution}) one can also compute the evolution equation for $\tau$:
\begin{equation}
\ddt(\tau)\w \vp=-\tau \w d\tau+\frac{1}{3}d(\|\tau\|^2_\vp) \w *_\vp\vp+\frac{1}{3}\|\tau\|^2_\vp\tau \w \vp -d*_\vp d\tau.
\end{equation}
From this one can deduce the evolution equations for $\tau_v$ and $\tau_h$. The resulting expressions are rather involved so we won't write them down here.
\begin{Rem}
	The evolution equations derived in this section generalise those derived in \cite{Fino2017} in the special case that $L^7=S^1\times P^6$ is a warped product. Note however that their choice of $SU(3)$-structure $(P^6,\check{\om},\check{\Om})$ differs from ours by a conformal factor so that $(\check{\om},\check{\Om})=(H\om,H^{3/2}\Om)$. In particular, $\check{\om}$ is not symplectic but on the other hand with respect to $g_{\check{\om}}$, instead of $g_\om$, equation (\ref{evolutionofH}) becomes parabolic. 
	Since the induced flow on the data $(H,\xi,\Om)$ is still generally quite complicated we shall only study it in a couple of simple examples in the last section, which exclude their case i.e. not warped products.
\end{Rem}

\subsection{The $S^1$-invariant soliton equation} Having derived the general $S^1$-invariant flow equations the natural next step is to work out the soliton equation. But before doing so we first prove a non-existence result in the compact case. By compact we  always mean without boundary.
\subsubsection{Non-existence of compact solitons with continuous symmetry}
\begin{Prop}\label{theorem1}
	Let $\vp$ be an exact $G_2$-structure on a compact $7$-manifold $L^7$ so that $\vp=d\beta$ for $\beta \in \Lm^2(L^7)$. Then $(L^7,\vp)$ admits no non-trivial infinitesimal symmetry i.e.
	\[\mathcal{L}_X \vp =0 \text{\ \ \ \ if and only if \ \  } X \equiv 0. \]
\end{Prop}
\begin{proof}
	Suppose that $\mathcal{L}_X \vp=0$ then
	\begin{align*}
		0& =\int_L d ((X \ip \vp) \w (X \ip \vp) \w \beta )\\ 
		&= \int_L (X \ip \vp) \w (X \ip \vp) \w \vp \\
		&=\int_L 6 g_\vp(X,X)\vol_\vp \geq 0
	\end{align*}
	The first equality is from Stokes' Theorem, the second follows from the fact that $d(X \ip \vp)=\mathcal{L}_X \vp=0$ and the last is a consequence of (\ref{metricg2}). Hence $\|X\|_\vp$ is identically zero i.e. $X\equiv 0$.
\end{proof}
\begin{Cor}\label{expandingsoliton}
	A non-torsion free Laplacian soliton on a compact manifold $(L^7,\vp)$ admits no non-trivial infinitesimal symmetry.
\end{Cor}
\begin{proof}
	Recall that the soliton equation is
	\[\Delta_\vp\vp=\lambda \vp + \mathcal{L}_V \vp\]
	for some $V \in \Gamma(TL^7)$, or equivalently
	\[\lambda \vp =d(\tau - V \ip \vp).\]
On a compact manifold, from \cite[Proposition 9.5]{Lotay2017} we know that $\lambda \geq 0$ with equality if and only if $\vp$ is torsion free. So we only need to consider the case when $\lambda>0$ and the result follows from the previous Proposition.
\end{proof}
This makes the construction of expanding solitons on compact manifolds quite a hard problem as one cannot use continuous symmetries to simplify the PDEs, so this suggests that one might have to use hard analysis to find examples (if any exist at all). 

More generally, let $\al:=\al^2_7 + \al^2_{14} \in \Lm^2\cong \Lm^2_7\oplus \Lm^2_{14}$ be a closed $2$-form on compact $L^7$ and $\vp =d \beta$. Then
\begin{equation}
	0=\int_L d(\al \w \al \w \beta)=\int_L \al \w \al \w \vp=\int_L \Big(2\|\al^2_7\|^2-\|\al^2_{14}\|^2 \Big)\vol_\vp. \label{blah}
\end{equation}
Observe that only the cohomology class of $\al$ is relevant here. In Proposition \ref{theorem1} we used the fact that if $\mathcal{L}_X \vp=0$ then we can set $\al = X \ip \vp \in \Lm^2_7$.  Note also that for any closed $G_2$-structure $\vp$ (not necessarily exact) we have the following one-to-one correspondences:
\begin{gather*}
\{X\in\Gamma(TL)\ |\ \mathcal{L}_X\vp=0 \}\xleftrightarrow{\al=X \ip \vp}\{\al\in \Lm^2_7\ | \ d\al=0 \} \xleftrightarrow{} \mathcal{H}^2_7:=\{\al\in \Lm^2_7\ | \ \Delta_\vp \al=0 \}\\
\{\al\in \Lm^2_{14}\ | \ d\al=0 \} \xleftrightarrow{} \mathcal{H}^2_{14}:=\{\al\in \Lm^2_{14}\ | \ \Delta_\vp \al=0 \}
\end{gather*}
Here we are using the fact that for $\al \in \Lm^2_7$, we  have that $2*_\vp \al=\al \w \vp$ and hence any closed $\al\in \Lm^2_7$ is also coclosed. Likewise the analogous argument applies for closed $\al \in \Lm^2_{14}$ since $*_\vp \al=-\al \w \vp$ . 

Although for \textit{torsion free} $G_2$-structures it is known that 
$$\mathcal{H}^2=\mathcal{H}^2_7 \oplus \mathcal{H}^2_{14}$$
where $\mathcal{H}^2$ denotes the space of harmonic $2$-forms \cite{Joycebook}, this is \textit{not} generally true for strictly closed $G_2$-structures. One way of seeing this is suppose $\al=\al^2_7 + \al^2_{14}\in \Lm^2_7 \oplus \Lm^2_{14}$ is a harmonic $2$-form on $(L^7,\vp)$ then we have
\[d(\al^2_{7})=-d(\al^2_{14})=-\frac{1}{7}g_\vp(\al,\tau)\vp + \gamma \in \langle \vp\rangle \oplus \Lm^3_{27}. \]
Thus, there could exist harmonic $2$-forms not in $\mathcal{H}^2_7 \oplus \mathcal{H}^2_{14}.$ In any case, we know that these spaces are finite dimensional vector spaces since from Hodge theory $\dim(\mathcal{H}^2)$ is the second Betti number of $L^7$.

So (\ref{blah}) says that in fact there does not exist any closed $2$-form strictly of type $\Lm^2_7$ or $\Lm^2_{14}$ on $(L^7,\vp=d\beta)$ and hence we deduce the following:
\begin{Th}\label{topoexact}
	If a compact manifold $L^7$ admits an exact $G_2$-structure $\vp=d\beta$ then $(L^7,\vp)$ has no non-trivial infinitesimal symmetry and it also does not have any closed (equivalently harmonic) $2$-form of pure type $14$ i.e. $\mathcal{H}^2_7 \oplus \mathcal{H}^2_{14}=0$. 
\end{Th}
In particular, this Theorem applies to compact expanding solitons. Another immediate consequence is that any $S^1$ bundle on $(L^7,\vp=d\beta)$ cannot admit a connection whose curvature $2$-form belongs to $\Lm^2_{14}$ or $\Lm^2_{7}$ i.e. it cannot be a $G_2$ instanton or anti-instanton respectively.
Note that it is still an important open problem whether a compact manifold can even admit an exact $G_2$-structure.
\begin{Rem}
In \cite{RafferoSymmetryofClosedG2} Podest\`a and Raffero used similar arguments to show that for closed $G_2$-structures on compact manifolds the Lie algebra of infinitesimal symmetry $\mathcal{H}^2_7$ is in fact abelian and of dimension at most $6$. They also exhibit an example showing that this bound is sharp. Theorem \ref{topoexact} shows that the exact case is very different from the closed one.
\end{Rem}

\subsubsection{The $S^1$-invariant soliton equation}
We now derive the equations for $S^1$-invariant Laplacian solitons in terms of the data on $P^6$. Note that in view of Corollary \ref{expandingsoliton} this only applies to the non-compact case. We now consider the soliton equation
\[-\lambda\vp= d(-\tau+V \ip \vp)\]
with $V\in \Gamma(TP^6)$ i.e. we assume that $V$ is a horizontal vector field on $L^7$. Then under the free $S^1$ action generated by the vector field $Y$ as before, this reduces to the pair
\begin{gather}
-\lambda \om = d(\tau_v + V \ip \om),\\
\lambda H^{3/2}\Om^+=d\tau_h + d\xi \w (\tau_v+ V \ip \om)-d(H^{3/2}V \ip \Om^+),
\end{gather}
with $\tau_v$ and $\tau_h$ as defined in section \ref{quotientofclosedg2structures}. Observe that if $\lambda \neq 0$ then the symplectic form $\om$ is necessarily exact (which implies that $P^6$ is non-compact and hence $L^7$ as we already noted), as it is for $\vp$.
If $V=\nabla f$ is the gradient vector field for some function $f$ on $(P^6,g_\om)$ then we can rewrite the soliton equation using Lemma \ref{nicelemma} and (\ref{tauhdef}) as
\begin{align}
	-\lambda \om &= dd^c(2H^{-1}-f) -2d(H^{-2}J\gamma_6^1),\\
	\lambda H^{3/2}\Om^+=\ &dd^{*_\om}(H^{1/2}\Om^+) + d\xi \w (d^c(2H^{-1}-f)-2H^{-2}J\gamma_6^1)\\ &-d*_\om(H^{3/2}d^c f \w \Om^+).\nonumber
\end{align}
We shall not attempt to solve this system here but we will give an example of a solution in the next section.
\section{Examples of $S^1$-invariant Laplacian flow}\label{examples}
\subsection{The Bryant-Fern\'andez example}\

The compact nilmanifold $L^7$ associated to the $2$-step nilpotent Lie algebra\footnote{Here we are using Salamon's notation \cite{Salamoncomplexstructures} to mean that the Lie algebra admits a coframing $e^i$ with $de^i=e^{jk}$, where $jk$ denotes the $i$th entry.} 
$$(0, 0, 0, 0, 0, 12, 13 )$$ 
admits a closed $G_2$-structure given by
\[\vp_0= e^{123}+e^{145}+e^{167}+e^{246}-e^{257}-e^{347}-e^{356}.\]
This example was discovered by Fern\'andez in \cite{Fernandezexample} and Bryant worked out the Laplacian flow on this example in \cite{Bryant06someremarks}. The solution to the Laplacian flow is given by
\[\vp_t= f^3e^{123}+e^{145}+e^{167}+e^{246}-e^{257}-e^{347}-e^{356} ,\]
where $f:=(\frac{10}{3} t+1)^{\frac{1}{5}}$ cf. \cite{Fernandez2016LF}. This solution is immortal and the volume grows as $\sim t^{1/5}$ in time. Bryant also showed that $L^7$ cannot admit a torsion free $G_2$-structure for topological reasons and hence one cannot expect the flow to converge. Nonetheless we have that $\| \tau_t \|^2_{\vp_t} = 2f^{-5}$ converges to zero as $t \to \infty$ and \cite[Theorem 4.2]{Fernandez2016LF} shows that $g_{\vp_t}$ converges in suitable sense to a flat metric.

Following the construction described in section \ref{quotientofclosedg2structures} we choose the vector field $Y$ generating an $S^1$ action preserving $\vp_0$ to be $e_6$ so that the connection form $\xi_0=e^6.$ The solution to the induced flow on the quotient nilmanifold $P^6$ is then given by
\begin{gather*}
H_t=f^{1/2},\\
\om_t=\om_0=-e^{17}+e^{24}-e^{35},\\
\Om_t^+= f^{\frac{9}{4}}e^{123} + f^{-\frac{3}{4}}(e^{145}-e^{257}-e^{347}),\\
\Om_t^-= -f^{-\frac{9}{4}}e^{457} - f^{\frac{3}{4}}(e^{237}+e^{125}+e^{134}),\\
g_{\om_t}=f^{\frac{3}{2}} (({e^1})^2+({e^2})^2+({e^3})^2)+f^{-\frac{3}{2}}(({e^4})^2+({e^5})^2+({e^7})^2),\\
\gamma^1_6=\frac{1}{2}f^{-5/2}e^5,\ \ d\xi=\frac{1}{2}(e^{12}-f^{-3}e^{47}) +\frac{1}{2}(e^{12}+f^{-3}e^{47}) \in \Lm^2_6\oplus\Lm^2_8
\end{gather*}
We see that the symplectic form, and hence the volume, stays constant while the metric (equivalently the complex structure) degenerates as $t\to \infty$. As in \cite{Fernandez2016LF} one can verify that the curvature decays to zero as $t \to \infty.$ 
Note that neither $\tau_6$ nor $\tau_8$ is zero in this case, so this example can be viewed as a generic case with regards to the type of the $SU(3)$-structure. 

\subsection{New examples from the Apostolov-Salamon Ansatz}\ 

As a flow on $SU(3)$-structures one can ask if the flow preserves any interesting geometric quantity. For instance we already saw that since $\vp$ stays closed under the flow $(P^6,\om)$ stays symplectic. A natural question to ask is: if the almost complex structure $J$ is initially integrable, does this persists under the flow? 

If $\vp$ is torsion free and $(P^6,\om,\Om^\pm)$ is K\"ahler then Apostolov and Salamon proved that $P^6$ is in fact a $\C^\times$ bundle over a $4$-manifold $M^4$, which in special cases turns out to be hyperK\"ahler cf. \cite[Theorem 1]{Apostolov2003}. This motivates us to search for solutions to the flow preserving the K\"ahler condition on these spaces. 

Consider the manifold $L^7=N^6\times \R_u$ where $N^6$ is a compact nilmanifold associated to the Lie algebra $$(0,0,0,0,13-24,14+23).$$ The $G_2$-structure determined by
\begin{equation}\vp = -f^2h (\om_1 \w du) + g^2 h(e^{56} \w du) - g f^2(\om_3 \w e^5 - \om_2 \w e^6),\label{varphiAS}\end{equation}
defines a $G_2$ coframing on $L^7$ given by $E^1=f e^3$, $E^2=f e^2$, $E^3=g e^5$, $E^4=-g e^6$, $E^5=-f e^1$, $E^6=-f e^4$ and $E^7= hdu$, where $f,g,h$ are (nowhere vanishing) functions of $u$ only and $\om_i$ denote the standard self-dual $2$-forms in $\langle e^1,e^2,e^3,e^4\rangle$ i.e.
\[\om_1=e^{12}+e^{34},\ \  \om_2=e^{13}-e^{24} \text{\ \ \ and\ \ \ } \om_3=e^{14}+e^{23}. \]
\begin{Lemma}\label{ASclosedcoclosed} The intrinsic torsion of $G_2$-structure defined by (\ref{varphiAS}) is given by 
\begin{enumerate}
\item\label{closedh} $d \vp=0$\ if and only if\ $\frac{\partial}{\partial u} (g f^2)= g^2h.$
\item  $d *_\vp \vp =0$\ if and only if     \ $ \frac{\partial}{ \partial u}(f g)=0$ and $ \frac{\partial}{\partial u} (f )= \frac{gh}{f}.$
\end{enumerate}
\end{Lemma}
\begin{proof}
	A direct calculation shows
	\begin{equation}
	d\vp=(g^2h-\ddu(gf^2))(\om_2 \w e^6-\om_3\w e^5)\w du
	\end{equation}
	and 
	\begin{equation}
		d(*_\vp\vp)=-\ddu(f^2g^2)(e^{56}\w \om_1)\w du+4f^2(f\ddu(f)-gh)e^{1234}\w du.\label{AStorsion}
	\end{equation}
\end{proof}
The explicit torsion free $G_2$-structure given by setting $f=(3u)^{1/3}$, $g=(3u)^{-1/3}$ and $h=1$ corresponds to Example 1 in \cite{Apostolov2003}.

Let us now impose that $\vp$ is closed, so that $h$ is determined by condition (\ref{closedh}) of Lemma \ref{ASclosedcoclosed}, and consider the $S^1$ action generated by the vector field $Y=e_6$. Then applying the construction of section \ref{quotientofclosedg2structures} one can compute that
\begin{gather*}
\om=(g^2 h) du \w e^5 + (g f^2) \om_2,\\
\Om^+= - (f^2 h g^{3/2}) \om_1 \w du - (g^{5/2}f^2) \om_3 \w e^5, \\
\Om^-= - (f^2 g^{5/2}) \om_1 \w e^5 + (g^{3/2}h f^2) \om_3 \w du \\
g_{\om}=g f^2 ((e^1)^2 + (e^2)^2+(e^3)^2+(e^4)^2 )+ g^3 (e^5)^2 + h^2 g (du)^2,\\
\xi=e^6,\ \ \  \gamma_6^1=-h f^{-2} du, \ \ d\xi = \om_3 \in \Lm^2_6,\\
H=g^{-1}, \ \ \ \pi_1= \frac{\partial}{\partial u} (\log (g^{5/2}f^2))du, \ \ \  \tau_8=0.
\end{gather*}
Since $\tau_8=(d\xi)^2_8=0$, from (\ref{omplustorsion}) and (\ref{omminustorsion}) we see that the only non-zero component of the $SU(3)$ torsion is $\pi_1$ and hence
it follows that these closed $G_2$-structures all admit K\"ahler reductions (see Definition \ref{su3strucdefinitions}).
\begin{Lemma}
The torsion form of closed $\vp$ is given by
\[ \tau= \ddu(f^2 g^2) \frac{1}{hg^2}\om_1 + 4(\frac{g^3}{f^2}-\frac{g^2}{fh}\ddu (f) )e^{56}.\]
\end{Lemma}
\begin{proof}
This follows directly from (\ref{AStorsion}) and using the closed condition i.e. (\ref{closedh}) of Lemma \ref{ASclosedcoclosed}.
\end{proof}
We now search for solutions to the Laplacian flow (\ref{LF}) of the form (\ref{varphiAS}), where only the functions $f,g,h$ depend on $t$. Computing the Laplacian flow gives the system
\begin{gather}  
	\ddt(f^2h) (\om_1 \w du)  = -\ddu ( \ddu (f^2 g^2) \cdot \frac{1}{hg^2}) (\om_1 \w du),\label{LF1'}\\ 
	 \ddt( g f^2)( \om_2 \w e^6-\om_3 \w e^5) = 4(\frac{g^3}{f^2}- \frac{g^2}{fh}\ddu(f) )(\om_2 \w e^6 - \om_3 \w e^5),\label{LF2'}\\ 
	 \ddt( g^2 h)(e^{56} \w du) =\ddu(4(\frac{g^3}{f^2} -\frac{g^2}{fh} \ddu(f)))( e^{56}\w du).\label{LF3'}
\end{gather} 
Equation (\ref{LF3'}) is a consequence of (\ref{LF2'}) by differentiating with respect to $u$ and using the closed condition from Lemma \ref{ASclosedcoclosed}. Thus, the Laplacian flow is reduced to the pair
\begin{gather}
\ddt (f^2 h)= -\ddu (\frac{1}{h g^2} \ddu(f^2 g^2)),\label{LF1}\\
\ddt (g f^2)= 4 g^2 ( \frac{g}{f^2}-\frac{1}{h f} \ddu f). \label{LF2}
\end{gather}
Since we are on a non-compact manifold the existence and uniqueness of a solution, given initial data, to (\ref{LF1}) and (\ref{LF2}) is not always guaranteed. We do however know that there exists at least one solution namely the (incomplete) torsion free one of Apostolov-Salamon \cite{Apostolov2003}.
Rather than addressing the general existence problem, we shall instead find another explicit solution as follows.\\
\noindent\textbf{A shrinking gradient soliton.}\\
\noindent With $f(u)=2^{-1/4}{\rm e}^{u/2},$ $g(u)=2^{1/2}{\rm e}^u$ and $h(u)=1$ we have
\[\vp_0 = -2^{-1/2}{\rm e}^u (\om_1 \w du) + 2 {\rm e}^{2u}(e^{56} \w du) - {\rm e}^{2u}(\om_3 \w e^5 - \om_2 \w e^6).\]
Taking $\lambda=-18$ and $V=15\cdot \partial_u$, one verifies directly that the soliton equation (\ref{solitonequationoriginal}) is satisfied. Thus, it defines a gradient shrinking soliton with the induced metric 
\[g_{\vp_0}=2^{-1/2}{\rm e}^{u}((e^1)^2+(e^2)^2+(e^3)^2+(e^4)^2)+2{\rm e}^{2u}((e^5)^2+(e^6)^2)+du^2\]
which is clearly complete. To the best of our knowledge this is the first example of an inhomogeneous shrinker.\\ 

To derive the general soliton equation we first observe that an invariant vector field $V$ is of the form $V=a\cdot \partial_u +b\cdot e_5+c\cdot e_6$, for functions $a(u)$, $b(u)$ and $c(u)$. Comparing with the expressions for $\tau$ and $\varphi$ it is easy to see that we get a consistent system only if $b=c=0$. By reparametrising the $u$-coordinate we can set $h=1$, and defining $F=f^2g$ and $G=g^2$ the closed condition becomes equivalent to $G=F'$.
We compute the soliton equation for the unknowns $(F(u),a(u))$ as
\begin{gather}
(\log(F^2 F'))'=\frac{\lambda}{(\log(F))'}+a,\\
\bigg(\frac{(F(F')^{1/2})'}{F'}\bigg)'=-\lambda F (F')^{-1/2}-(aF(F')^{-1/2})'.
\end{gather}
With the ansatz $F={\rm e}^{ku}$, we find the solution $\lambda=-\frac{9}{2}k^2$ and $a=\frac{15}{2}k>0$. The scalar curvature is $$Scal(g_{\vp_0})=-\frac{1}{2}\|\tau\|^2_{\vp_0}=-\frac{27}{4}k^2.$$
Observe that this construction applies to any hyperK\"ahler $4$-manifold $M^4$ such that $[\om_2], [\om_3] \in H^2(M^4,2\pi\mathbb{Z})$. In this case $P^6$ can be taken to be the total space of the $\mathbb{T}^2$ bundle determined by these integral cohomology classes and $e^5, e^6$ are the connection $1$-forms with curvature $\om_2, \om_3$ respectively \cite{Apostolov2003}. 

\begin{Rem}
We should also point out that recently Ball found the first examples of inhomogeneous \textit{steady} solitons on the same spaces \cite[Example 2]{GavinBall2020}, also arising from what we referred to as the Apostolov-Salamon Ansatz. \end{Rem}

\bibliography{referencing}
\bibliographystyle{plain}

\end{document}